\documentclass[11pt]{article}
\usepackage{amsthm,amsmath,amsfonts,latexsym,amssymb,mathrsfs,ulem}
\usepackage{color}

\theoremstyle{definition}
\newtheorem{Def}{Definition}[section]

\newtheorem{thm}[Def]{Theorem}

\newtheorem{lem}[Def]{Lemma}

\newtheorem{rem}[Def]{Remark}

\newtheorem{prob}[Def]{Problem}

\renewcommand{\thefootnote}{\fnsymbol{footnote}}

\definecolor{darkgreen}{rgb}{.1,.5,0}

\begin{document}

\title{Subnormal block Toeplitz operators}

\author{Mankunikuzhiyil Abhinand, Ra{\'u}l\ E.\ Curto, In Sung Hwang, Woo Young Lee,\\
 and Thankarajan Prasad}

\date{}

\maketitle

%%%%%%%%%%%%%%%%%%%%%%%%%%%%%%%%%%%%%%%%%%%%%%%%%%%%%%%%%%%%%%%%%%%%%%%%%%%%%%%%%
%
%
%              Abstract
%
%
%%%%%%%%%%%%%%%%%%%%%%%%%%%%%%%%%%%%%%%%%%%%%%%%%%%%%%%%%%%%%%%%%%%%%%%%%%%%%%%%%

\noindent{\bf Abstract.} In this paper we consider the subnormality
of block Toeplitz operators $T_\Phi$, where $\Phi$ is an $n\times n$
matrix-valued function on the unit circle $\mathbb T$ of the form
$$
\Phi=Q\Phi^* \quad \hbox{($Q$ is a finite Blaschke--Potapov
product).}
$$
This is related to a matrix-valued version of Halmos's Problem 5 and
Nakazi-Takahashi Theorem. \ We ask whether $T_\Phi$ is either normal or analytic if $T_\Phi$ is subnormal,
where $\Phi$ is of the above form. \ We give answers to this problem
for different cases of the symbol. \ Moreover, we provide a sufficient
condition for the answer to be affirmative when $\Phi^*$ is not of
bounded type.

\medskip

%%%%%%%%%%%%%%%%%%%%%%%%%%%%%%%%%%%%%%%%%%%%%%%%%%%%%%%%%%%%%%%%%%%%%%%%%%%%%%%

\setcounter{page}{1}
%%\tableofcontents

%%%%%%%%%%%%%%%%%%%%%%%%%%%%%%%%%%%%%%%%%%%%%%%%%%%%%%%%%%%%%%%%%%%%%%%%%%%%%%%%%%%%%%%%%%%%%%

\renewcommand{\thefootnote}{}

\bigskip

%%\tableofcontents

%%\vskip 1cm

%%%%%%%%%%%%%%%%%%%%%%%%%%%%%%%%%%%%%%%%%%%%%%%%%%%%%%%%%%%%%%%%%%%%%%%%%%%%%%%%%%%%%%%%%%
%
%
%                    Section 1
%
%
%%%%%%%%%%%%%%%%%%%%%%%%%%%%%%%%%%%%%%%%%%%%%%%%%%%%%%%%%%%%%%%%%%%%%%%%%%%%%%%%%%%%%%%%%%%

\section{Introduction}

\medskip

Throughout the paper, let $\mathcal H$ be a complex separable
Hilbert space and let $\mathcal B(\mathcal {H})$ be the algebra of
all bounded linear operators on $\mathcal H$. \ We denote by
$\hbox{Ran}\,T$, $\hbox{ker}\,T$ and $\hbox{rank}\,T$ the range, the
kernel, and the rank of $T\in\mathcal{B(H)}$, respectively. \ For
operators $A,B\in\mathcal{B(H)}$, let $[A,B]:=AB-BA$. \ An operator
$T\in\mathcal{B(H)}$ is called {\it normal} if $[T^*,T]=0$, {\it
hyponormal} if $[T^*,T]\ge 0$, and {\it subnormal} if $T$ has a
normal extension $N$, i.e., $T$ is a restriction of $N$ to $\mathcal
H$, where $N$ is a normal operator on some Hilbert space $\mathcal
{K}\supseteq \mathcal {H}$ and $\mathcal H$ is invariant for $N$.
The notion of subnormal operator was introduced by P.R. Halmos
\cite{Ha3} and has been much-studied in the ensuing decades. \ Also, the
class of hyponormal operators includes the classes of normal and
subnormal operators \cite{Ha1}.

Let $\mathbb T$ denote the unit circle in the complex plane $\mathbb
C$. \ A function $\theta\in H^\infty(\mathbb T)$ is called an {\it
inner} function if it satisfies $|\theta|=1$ a.e. on $\mathbb T$.
For a function $\varphi\in L^\infty(\mathbb T)$, the Toeplitz
operator $T_\varphi$ with symbol $\varphi$ on the Hardy space
$H^2(\mathbb T)$ is defined by
$$
T_\varphi f:=P(\varphi f),
$$
where $f\in H^2(\mathbb T)$ and $P$ is the orthogonal projection of
$L^2(\mathbb T)$ onto $H^2(\mathbb T)$. \ A Toeplitz operator
$T_\varphi$ is called {\it analytic} if $\varphi\in H^\infty(\mathbb
T)$. \ A function $\varphi\in L^\infty(\mathbb T)$ is said to be of
{\it bounded type} if $\varphi=\frac{\phi_1}{\phi_2}$, where
$\phi_1$ and $\phi_2$ are in $H^\infty(\mathbb T)$. \ A. Brown and
P.R. Halmos \cite{BH} characterized the normality of Toeplitz
operators in terms of their symbols. \ Also the hyponormality of
Toeplitz operators was characterized by C. Cowen \cite{Cow}:

\medskip

\noindent{\bf Cowen's Theorem} (\cite{Cow}, \cite{NT}). \ For each $\varphi\in
L^\infty(\mathbb T)$, the Toeplitz operator $T_\varphi$ is
hyponormal if and only if $\mathcal{E}(\varphi)$ is nonempty, where
$$
\mathcal{E}(\varphi)=\{k\in H^\infty(\mathbb T): \|k\|_\infty\le 1\
\hbox{and}\ \varphi-k\overline\varphi\in H^\infty\}.
$$
By contrast, it seems to be very hard to determine the subnormality
of Toeplitz operators in terms of their symbols. \ Indeed, P.R. Halmos
\cite{Ha1} posed the following problem:
$$
\hbox{Is every subnormal Toeplitz operator either normal or
analytic\,?}
$$
In 1984, C. Cowen and J. Long \cite{CoL} answered this problem negatively. \ However, a complete characterization of subnormal Toeplitz operators
by the properties of their symbols is still unknown. \ This naturally
raised the following problem:
$$
\hbox{Which Toeplitz operators are subnormal\,?}
$$
To date, there are two interesting partial answers to the Halmos'
Problem as in the below.

\bigskip

\noindent {\bf Abrahamse's Theorem} (\cite[Theorem]{Ab}). \ Let
$\varphi\in L^\infty$ such that $\varphi$ or $\overline\varphi$ is
of bounded type. \ If $T_\varphi$ is hyponormal and $\ker
[T_\varphi^*, T_\varphi]$ is invariant under $T_\varphi$, then
$T_\varphi$ is normal or analytic. \ In particular, in this case, if
$T_\varphi$ is a subnormal operator then $T_\varphi$ is either
normal or analytic.

\bigskip

\noindent{\bf Nakazi-Takahashi Theorem} (\cite[Theorem 15]{NT}). \ If
$T_\varphi$ is subnormal and $\varphi = q \overline\varphi$, where
$q$ is a finite Blaschke product, then $T_\varphi$  is either normal
or
 analytic.

\bigskip

Now let us turn our attention to the block Toeplitz operators. \ The intensive and fruitful study of block Toeplitz
operators has contributed to the development of various fields of
mathematics and physics (cf. \cite{BEV}, \cite{BD}, \cite{BS},
\cite{CHL3}, \cite{CHL4}).

Let $I_n$ be the $n \times n$ identity matrix, $M_{n\times m}$
denote the set of all $n\times m$ complex matrices and $M_n\equiv
M_{n\times n}$. \ As denote by $L^2_{\mathbb C^n}$, $H^2_{\mathbb
C^n}$ and $L^\infty_{M_n}$ the set of all $\mathbb C^n$-valued
Lebesgue square integrable functions on $\mathbb T$, the associated
Hardy space and the set of all $M_n$-valued essentially bounded
functions on $\mathbb T$, respectively. \ For a function $\Phi\in
L^\infty_{M_n}$, the block Toeplitz operator $T_\Phi$ with symbol
$\Phi$ is defined by $T_\Phi f :=P_n(\Phi f)$ for $f\in H^2_{\mathbb
C^n}$, where $P_n$ is the orthogonal projection of $L^2_{\mathbb
C^n}$ onto $H^2_{\mathbb C^n}$. \ Similarly, a block Hankel operator
$H_\Phi$ with symbol $\Phi\in L^\infty_{M_n}$ is defined by $H_\Phi
f:=J_n P_n^\perp(\Phi f)$ for $f\in H^2_{\mathbb C^n}$, where
$P_n^{\perp}$ is the orthogonal projection of $L^2_{\mathbb C^n}$
onto $(H^2_{\mathbb C^n})^{\perp}\equiv L^2_{\mathbb C^n}\ominus
H^2_{\mathbb C^n}$ and $J_n$ is the unitary operator from
$L^2_{\mathbb C^n}$ onto $L^2_{\mathbb C^n}$ defined by
$J_n(f)(z):=\overline zI_n f(\overline z)$ for $f\in L^2_{\mathbb
C^n}$. \ If we take $H^2_{\mathbb C^n}=H^2(\mathbb T)\oplus
\cdots\oplus H^2(\mathbb T)$, then we can easily see that
$$
T_\Phi=\begin{bmatrix}T_{\phi_{11}}&\hdots&T_{\phi_{1n}}\\
&\vdots\\
T_{\phi_{n1}}&\hdots&T_{\phi_{nn}}
\end{bmatrix}\quad\hbox{and}\quad
H_\Phi=\begin{bmatrix}H_{\phi_{11}}&\hdots&H_{\phi_{1n}}\\
&\vdots\\
H_{\phi_{n1}}&\hdots&H_{\phi_{nn}}
\end{bmatrix},
$$
where
$$
\Phi=\begin{bmatrix}{\phi_{11}}&\hdots& {\phi_{1n}}\\
&\vdots\\
{\phi_{n1}}&\hdots&{\phi_{nn}}
\end{bmatrix}.
$$
For a function $\Phi\in L^\infty_{M_{n\times m}}$, write
$$
\breve{\Phi}(z):=\Phi(\overline{z}) \quad \hbox{and} \quad
\widetilde\Phi:=\breve{\Phi}^*.
$$
A function $\Phi\in L^\infty_{M_{n\times m}}$ is of bounded type if
each of its entries is of bounded type. \ A function $\Theta\in
H^\infty_{M_{n\times m}}$ is called an inner function if
$\Theta^*\Theta=I_m$ a.e. on $\mathbb T$. \ The following basic
properties for Toeplitz and Hankel operators are implicitly used in
the sequel:
\begin{equation}\label{basic}
 \aligned
&T_\Phi^*=T_{\Phi^*},\ \  H_\Phi^*= H_{\widetilde \Phi} \quad
(\Phi\in
L^\infty_{M_n});\\
&T_{\Phi\Psi}-T_\Phi T_\Psi = H_{\Phi^*}^*H_\Psi \quad
(\Phi,\Psi\in L^\infty_{M_n});\\
&H_\Phi T_\Psi = H_{\Phi\Psi},\ \
H_{\Psi\Phi}=T_{\widetilde{\Psi}}^*H_\Phi \quad (\Phi\in
L^\infty_{M_n}, \Psi\in H^\infty_{M_n});\\
&H_\Phi^* H_\Phi - H_{\Theta \Phi}^* H_{\Theta\Phi} =H_\Phi^*
H_{\Theta^*}H_{\Theta^*}^*H_\Phi \quad (\Theta\in H^\infty_{M_n}
 \ \hbox{inner,}  \ \Phi\in L^\infty_{M_n}).
\endaligned
\end{equation}
For a function $\Phi\in L^2_{M_n}$, write
$$
\Phi_+:=\mathbb{P}_n (\Phi)\in H^2_{M_n} \quad \hbox{and} \quad
\Phi_-:=[\mathbb{P}_n^{\perp} (\Phi)]^*\in zH^2_{M_n},
$$
where $\mathbb{P}_n$ and $\mathbb{P}_n^{\perp}$ denote the
orthogonal projections form $L^2_{M_n}$ onto $H^2_{M_n}$ and
$(H^2_{M_n})^{\perp}\equiv L^2_{M_n} \ominus H^2_{M_n}$,
respectively. \ Then we may write
$$
\Phi=\Phi_-^*+\Phi_+.
$$
For $\Phi\in L^\infty_{M_n}$, the following identities hold:
$$
P_nJ_n=J_nP_n^\perp, \quad P_n^\perp J_n=J_n P_n  \quad \hbox{and}
\quad J_n M_{\Phi}=M_{\breve{\Phi}}J_n.
$$
An $n\times n$ matrix-valued function $Q$ is called a {\it finite
Blaschke-Potapov product} if $Q$ is of the form
$$
Q(z)=v\prod_{m=1}^M (b_m(z)P_m+(I-P_m)),
$$
where $v$ is an $n\times n$ unitary constant matrix, $b_m$ is a {\it
Blaschke factor} of the form
$$
b_m(z)=\frac{z-\alpha_m}{1-\overline\alpha_m z}\ \
(\alpha_m\in\mathbb D)
$$
and $P_m$ is an orthogonal projection in $\mathbb  C^n$.

In 2006, C. Gu, J. Hendricks and D. Rutherford \cite{GHR} characterized normality
and hyponormality of block Toeplitz operators.

\medskip

\noindent{\bf Hyponormality and normality of block Toeplitz
operators} (\cite{GHR}).

\medskip
\noindent (a) (Hyponormality) For each $\Phi\in L^\infty_{M_n}$, let
$$
\mathcal{E}(\Phi)=\bigl\{K\in H^\infty_{M_n}: \|K\|_\infty\le 1\
\hbox{and}\ \Phi-K\Phi^*\in H^\infty_{M_n}\bigr\}.
$$
Then $T_\Phi$ is hyponormal if and only if $\Phi$ is normal (i.e.,
$\Phi^*\Phi=\Phi\Phi^*$ a.e. on $\mathbb T$) and $\mathcal{E}(\Phi)$
is nonempty.

\medskip
\noindent (b) (Normality) Let $\Phi=\Phi_-^* +\Phi_+$ be normal. \ If
$\hbox{det}\,\Phi_+$ is not identically zero then $T_\Phi$ is normal
if and only if $\Phi$ is normal and $\Phi_+-\Phi_+(0)=\Phi_-U$ for some constant unitary
matrix $U$.

\bigskip

For a function $\Phi\in H^2_{M_{n\times r}}$, we say that $\Delta\in
H^2_{M_{n\times m}}$ is a {\it left inner divisor} of $\Phi$ if
$\Phi=\Delta A$ for some $A\in H^2_{M_{m\times r}}$ ($m\le n$) and
$\Delta$ is an inner function. \ Two functions $\Phi\in
H^2_{M_{n\times r}}$ and $\Psi\in H^2_{M_{n\times m}}$ are said to
be {\it left coprime} if the only common left inner divisor of both
$\Phi$ and $\Psi$ is a constant unitary matrix, and {\it right
coprime} if $\widetilde\Phi$ and $\widetilde\Psi$ are left coprime.
Two functions $\Phi$ and $\Psi$ in $H^2_{M_n}$ are said to be {\it
coprime} if they are both left and right coprime. \ If $\Phi\in
H^2_{M_n}$ is such that $\hbox{det}\,\Phi\ne 0$,
 then any left inner divisor $\Delta$ of $\Phi$ is
square, i.e., $\Delta\in H^2_{M_n}$\label{square2}. \ If $\Phi\in
H^2_{M_n}$ is such that $\hbox{det}\,\Phi\ne 0$, then we say that
$\Delta\in H^2_{M_{n}}$ is a {\it right inner divisor} of $\Phi$ if
$\widetilde{\Delta}$ is a left inner divisor of $\widetilde{\Phi}$
(\cite{CHL4}).

\bigskip

In \cite{CHL2}, Curto-Hwang-Lee extended Abrahamse's Theorem to the
case of block Toeplitz operators.

\begin{lem} ({\bf Abrahamse's Theorem for Matrix-Valued
Symbols}) \cite[Theorem 3.5]{CHL2} \label{lem1000} \ \ Suppose
$\Phi=\Phi_-^*+\Phi_+\in L^\infty_{M_n}$ is such that $\Phi$ and
$\Phi^*$ are of bounded type of the form
$\Phi_+=A^*\Theta_0\Theta_2$ and $\Phi_-=B^*\Theta_2$, where
$\Theta_i=\theta_i I_n$ with an inner function $\theta_i$ ($i=0,2$)
and $A,B\in H^2_{M_n}$. \ Assume that $A$ and $\Theta_2$ are left
coprime and $B$ and $\Theta_2$ are left coprime. \ If
\begin{itemize}
\item[(i)] $T_\Phi$ is hyponormal and
\item[(ii)] $\ker[T_\Phi^*, T_\Phi]$ is invariant under
$T_\Phi$,
\end{itemize}
then $T_\Phi$ is either normal or analytic. \ In particular, if
$T_\Phi$ is subnormal then it is either normal or analytic.
\end{lem}

\medskip

Also, in the context of the Nakazi-Takahashi Theorem, Curto-Hwang-Lee \cite{CHL1} posed the following matrix-valued version:

\begin{prob} \cite[Problem 6.2]{CHL1}\label{prob1001}
If $\Phi\in L^\infty_{M_n}$ is such that $T_\Phi$ is a subnormal
operator and
$$
\Phi=Q\Phi^*\quad \hbox{($Q$ is a finite Blaschke-Potapov product)},
$$
does it follow that $T_\Phi$ is either normal or analytic\,?
\end{prob}

\medskip

The aim of this paper is to give an answer to Problem
\ref{prob1001}.

\bigskip

%%%%%%%%%%%%%%%%%%%%%%%%%%%%%%%%%%%%%%%%%%%%%%%%%%%%%%%%%%%%%%%%%%%%%%%%%%%%%%%%
%
%
%                Section 2
%
%
%%%%%%%%%%%%%%%%%%%%%%%%%%%%%%%%%%%%%%%%%%%%%%%%%%%%%%%%%%%%%%%%%%%%%%%%%%%%%%%

\section{The statement of main results}

\medskip

We now analyze Problem \ref{prob1001} by splitting it into cases
following the features of the symbols.

\medskip

Let $\Phi\in L^\infty_{M_n}$ be such that
\begin{equation}\label{setting}
\Phi=Q\Phi^*\quad\hbox{for a finite Blaschke-Potapov product $Q$.}
\end{equation}
There are three cases to consider.

\medskip

\begin{itemize}
\item[] {\bf Case 1}: $\Phi^*$ is of bounded type of the form
$$
\Phi_-=B^*\Theta\quad \hbox{(where $B\in H^2_{M_n}$, $\Theta=\theta
I_n$ for $\theta$ an inner function),}
$$
where $B$ and $\Theta$ are left coprime.

\item[] {\bf Case 2}: $\Phi^*$ is of bounded type of the form
$$
\Phi_-=B^*\Theta\quad \hbox{(where $B\in H^2_{M_n}$, $\Theta=\theta
I_n$ for $\theta$ an inner function),}
$$
where $B$ and $\Theta$ are {\it not} left coprime.

\item[] {\bf Case 3}: $\Phi^*$ is not of bounded type.
\end{itemize}

\bigskip

Under the assumption (\ref{setting}), if  $\Phi^*$ is of bounded type
then it is easy to show that $\Phi$ is also of bounded type. \ For
Case 1, we give an affirmative answer to Problem \ref{prob1001}. \ To
do so, we need a preliminary observation. \ Suppose $\Phi:=\Phi_-^*+
\Phi_+\in L^\infty_{M_n}$ is such that $\Phi$ and $\Phi^*$ are of
bounded type. \ It is known (cf. \cite[Lemma 3.1]{CHL2})
that if $T_\Phi$ is hyponormal then we may write
\begin{equation}\label{btf}
\Phi_+=A^*\Theta_0\Theta_2\quad\hbox{and}\quad \Phi_- = B^*\Theta_2,
\end{equation}
where $\Theta_i =\theta_i I_n$ with an inner function $\theta_i$
($i=0,2$) and $A, B\in H^{2}_{M_n}$.

\medskip

In fact we can prove more:

\medskip

\begin{thm}\label{thm2001}
Suppose $\Phi:=\Phi_-^*+ \Phi_+\in L^\infty_{M_n}$ is such that
$\Phi$ and $\Phi^*$ are of bounded type. \ In view of (\ref{btf}), we
may write
$$
\Phi_- = B^*\Theta_2,
$$
where $\Theta_2 =\theta_2 I_n$ with an inner function $\theta_2$ and
$B\in H^{2}_{M_n}$. \ Assume that $B$ and $\Theta_2$ are left coprime.
If
\medskip

{\rm (i)} $T_\Phi$ is hyponormal; and

{\rm (ii)} $\text{\rm ker}\,[T_{\Phi}^*, T_{\Phi}]$ is invariant
under $T_{\Phi}$
\medskip

\noindent then $T_{\Phi}$ is normal or analytic. \ Hence, in
particular, if $T_\Phi$ is subnormal then it is either normal or
analytic.
\end{thm}

\bigskip

For Case 2, we give a negative answer to Problem \ref{prob1001}. \ To
see this, let
$$
\Phi:=\begin{bmatrix}z+\overline z&0\\ 0&z\end{bmatrix}.
$$
Then $T_\Phi$ is subnormal because $T_{z+\overline z}$ is
self-adjoint and $T_z$ is subnormal. \ But clearly, $T_\Phi$ is
neither analytic nor normal because $T_z$ is not normal. \ Observe
$$
\Phi=Q\Phi^*\quad\hbox{for a finite Blaschke-Potapov product
$Q:=\begin{bmatrix}1&0\\ 0&z^2\end{bmatrix}$}
$$
and
$$
\Phi_-=\begin{bmatrix}z&0\\ 0&0\end{bmatrix}=B^*\Theta\quad
\hbox{with}\ \ B:=\begin{bmatrix}1&0\\0&0\end{bmatrix}\ \
\hbox{and}\ \  \Theta:=\begin{bmatrix}z&0\\0&z\end{bmatrix}.
$$
Note that $B$ and $\Theta$ are not (left) coprime (see \cite[Lemma
3.3]{CHL2}).

\bigskip

For Case 3, we also give a negative answer to  Problem
\ref{prob1001}. \ To see this, let
$$
\Phi:=\begin{bmatrix}f+\overline f&0\\ 0&z\end{bmatrix},
$$
where $f$ is not of bounded type. \ Then $\Phi^*$ is not of bounded
type and by the same argument as in Case 2 we can see that $
\Phi=Q\Phi^*$ for some finite Blaschke-Potapov product and $T_\Phi$
is subnormal, but neither analytic nor normal because $T_z$ is
subnormal but not normal.

\bigskip

However, we give a sufficient condition for the answer to be
affirmative in Case 3.

\bigskip

For $\Phi=[\phi_{ij}]\in L^{\infty}_{M_n}$, write
$$
\overline{\Phi}:=[\overline{\phi}_{ij}].
$$
Then we have:

\begin{thm}\label{thm2002}
Let $\Phi=Q\Phi^*$ be a matrix-valued function in $L^\infty_{M_n}$
such that $\Phi^*$ is not of bounded type, where $Q$ is a finite
Blaschke-Potapov product. \ Assume that $T_{\Phi^*}$ and
$H_{\overline{\Phi}}$ are both injective. \ If $T_\Phi$ is subnormal
then it is normal.
\end{thm}

\medskip

In Section 3, we give a proof of Theorem \ref{thm2001} and Section 4
is devoted to a proof of Theorem \ref{thm2002}.

\bigskip

%%%%%%%%%%%%%%%%%%%%%%%%%%%%%%%%%%%%%%%%%%%%%%%%%%%%%%%%%%%%%%%%%%%%%%%%%%%%%%%%
%
%
%                Section 3
%
%
%%%%%%%%%%%%%%%%%%%%%%%%%%%%%%%%%%%%%%%%%%%%%%%%%%%%%%%%%%%%%%%%%%%%%%%%%%%%%%%

\section{The proof of Theorem \ref{thm2001}}

\medskip
We note that the representation (\ref{btf}) is ``minimal," in the sense that
if $\omega I_n$ ($\omega$ is inner) is a common inner divisor of
$\Theta_0\Theta_2$ and $A$, or of $\Theta_2$ and $B$, then $\omega$ is
constant. \ On the other hand, (\ref{btf}) may be expressed
as in the form of the right coprime factorization as
\begin{equation}\label{btf2}
\Phi_+=\Delta_2\Delta_0A_r^*\quad\hbox{and}\quad \Phi_-=\Delta_2
B_r^*,
\end{equation}
where $\Delta_2\Delta_0$ and $A_r$ are right coprime and so are
$\Delta_2$ and $B_r$ (cf.
\cite[Lemma 3.2]{CHL2}).

\medskip

The {\it shift} operator $S$ on $H^2_{\mathbb C^n}$ is defined by $
S:=T_{z I_n}$. \ The following fundamental result known as the
Beurling-Lax Theorem is used in the sequel.
\bigskip

\noindent {\bf The Beurling-Lax Theorem.} \cite{Be}, \cite{Ha},
\cite{La} A nonzero subspace $M$ of $H^2_{\mathbb C^n}$ is invariant under
the shift operator $S$ on $H^2_{\mathbb C^n}$ if and only if
$M=\Theta H^2_{\mathbb C^m}$, where $\Theta$ is an inner matrix
function in $H^{\infty}_{M_{n\times m}}$ ($m\le n$). \ Furthermore,
$\Theta$ is unique up to a unitary constant right factor; that is,
if $M=\Delta H^2_{\mathbb{C}^r}$ \hbox{\rm (}for $\Delta$ an inner
function in $H^{\infty}_{M_{n\times r}}$\hbox{\rm )}, then $m=r$ and
$\Theta=\Delta W$, where $W$ is a \hbox{\rm (}constant in z\hbox{\rm
)} unitary matrix mapping $\mathbb C^m$ onto $\mathbb C^m$.

\bigskip

We are ready for:

\medskip

\begin{proof} [Proof of Theorem \ref{thm2001}]
Let $\Phi=\Phi_-^*+ \Phi_+\in L^\infty_{M_n}$ be such that $\Phi$
and $\Phi^*$ are of bounded type. \ Suppose that $T_\Phi$ is
hyponormal and $\text{\rm ker}\,[T_{\Phi}^*, T_{\Phi}]$ is invariant
under $T_{\Phi}$. \ In view of (\ref{btf}), we may write
$$
\Phi_+=A^*\Theta_0\Theta_2\quad\hbox{and}\quad \Phi_- = B^*\Theta_2,
$$
where $\Theta_i:=\theta_i I_n$  with an inner function $\theta_i$
($i=0,2$). \ In view of Lemma \ref{lem1000}, it suffices to prove that
if $B$ and $\Theta_2$ are left coprime then $A$ and $\Theta_2$ are
left coprime.

By recalling that left coprime-ness and the right coprime-ness
coincide for matrix-valued functions $A\in H^2_{M_n}$ and $\theta_2
I_n\equiv \Theta_2$ (where $\theta_2$ is an inner function) (cf.
\cite[Lemma 3.3]{CHL2}, \cite[Corollary C.14]{CHL5}), in view of
(\ref{btf}) and (\ref{btf2}), we can write
$$
\Phi_+ = A^*\Theta_0 \Theta_2=\Theta_2 \Delta_1 A_r^* \quad
\hbox{and} \quad \Phi_-=\Theta_2B^*,
$$
where $\Delta_1\in H^{\infty}_{M_n}$ is an inner function, $A_r\in
H^2_{M_n}$ and $\Theta_2\Delta_1$ are right coprime. \ First of all, a
careful analysis for STEP 1 of the proof for \cite[Theorem
3.5]{CHL2} shows that
\begin{equation}\label{3.16-10}
\Theta_0 H^2_{\mathbb C^n}\ \subseteq\ \hbox{\rm ker}\,[T_{\Phi}^*,
T_{\Phi}].
\end{equation}
(Note that we didn't employ the assumption ``$A$ and $\Theta_2$ are
left coprime" to get (\ref{3.16-10}) in STEP 1 of the proof for
\cite[Theorem 3.5]{CHL2}.) We first show that
\begin{equation}\label{3.16-8}
\hbox{$\Theta_0$ and $\Theta_2$ are coprime.}
\end{equation}
To see this we assume to the contrary that $\Theta_0$ and $\Theta_2$
are not coprime. \ Then there exists an inner function $\Omega$ of the
form $\Omega:=\omega I_n$ (where $\omega$ is a nonconstant inner
function) such that
\begin{equation}\label{3.16-6}
\Theta_0=\Omega \Theta_0^{\prime}\quad\hbox{and}\quad
\Theta_2=\Omega \Theta_2^{\prime}.
\end{equation}
Indeed, by the Beurling-Lax Theorem, there exists an inner function
$\Omega\in H^\infty_{M_n}$ such that $\Omega H^2_{\mathbb
C^n}=\Theta_0 H^2_{\mathbb C^n}\bigvee \Theta_2 H^2_{\mathbb C^n}$
and moreover,
$$
\Omega H^2_{\mathbb C^n} = \bigoplus_{j=1}^n
\bigl(\theta_0 H^2\bigvee \theta_2 H^2\bigr) =\bigoplus_{j=1}^n
\omega H^2,
$$
where $\omega$ is a common inner divisor of $\theta_0$
and $\theta_2$. \ Thus $\Omega=\omega I_n$, which shows
(\ref{3.16-6}). \ Then since
$\Theta_2\Theta_0^{\prime}=\Omega\Theta_0^{\prime}\Theta_2^{\prime}=\Theta_0\Theta_2^{\prime}$,
it follows from (\ref{3.16-10}) that
$$
\Theta_2\Theta_0^{\prime} H_{\mathbb C^n}^2\subseteq \Theta_0
H_{\mathbb C^n}^2\subseteq \hbox{\rm ker}\,[T_{\Phi}^*, T_{\Phi}].
$$
Observe that
\begin{equation}\label{3.9-1}
[T_{\Phi}^* , T_{\Phi}]=
H_{\Phi_+^*}^*H_{\Phi_+^*}-H_{\Phi_-^*}^*H_{\Phi_-^*}
=H_{A\Theta_2^*\Theta_0^*}^* H_{A\Theta_2^*\Theta_0^*} -
H_{B\Theta_2^*}^* H_{B\Theta_2^* },
\end{equation}
which implies
$$
H_{A\Theta_2^*\Theta_0^*} (\Theta_2\Theta_0^{\prime} H_{\mathbb
C^n}^2)=\{0\},\quad \hbox{so that}\ \ H_{A\Omega^*}H_{\mathbb
C^n}^2=\{0\}.
$$
Thus we must have that $G\equiv A\Omega^* \in H_{M_n}^2$. \ Then we
can write
$$
\Phi_+=\Theta_0\Theta_2 A^*=\Theta_0\Theta_2^\prime \Omega A^*
=\Theta_0\Theta_2^\prime G^*,
$$
which is a contradiction because the representation
$\Phi_+=\Theta_0\Theta_2 A^*$ is in ``minimal" form in view of
(\ref{btf}). \ This proves (\ref{3.16-8}).

We next show that $A$ and $\Theta_2$ are (right) coprime, To see
this, let $\Delta^{\prime}$ be a common right inner divisor of $A$
and $\Theta_2$. \ Then we can write
$$
A= A_1\Delta^{\prime}\quad\hbox{and}\quad \Theta_2=\Delta
\Delta^{\prime},
$$
where $\Delta \in H_{M_n}^{\infty}$ is inner and $A_1 \in
H_{M_n}^2$. \ We thus have
$$
\Phi_+ = \Theta_0 \Theta_2A^* =\Theta_0\Delta A_1^*=\Delta \Theta_0
A_1^*.
$$
But since
$$
\Phi_+ =\Theta_2 \Delta_1 A_r^* \quad(\hbox{where $\Theta_2\Delta_1$
and $A_r$ are right coprime}),
$$
it follows that $\Theta_2 \Delta_1$ is a left inner divisor of
$\Delta \Theta_0$. \ Thus $\Theta_2$ is a left inner divisor of
$\Delta \Theta_0$. \ We now claim
\begin{equation}\label{3.17}
\hbox{$\Theta_2$ is a left inner divisor of $\Delta$.}
\end{equation}
Indeed, since by (\ref{3.16-8}), $\Theta_0$ and $\Theta_2$ are
coprime, it follows from the Beurling-Lax Theorem that $\Theta_0
H_{\mathbb C^n}^2 \bigcap \Theta_2 H_{\mathbb C^n}^2=\Theta_0 \Theta_2
H_{\mathbb C^n}^2$. \ Therefore, if $\Theta_2$ is a left inner divisor of
$\Delta \Theta_0$, then $\Delta \Theta_0 H_{\mathbb C^n}^2 \subseteq
\Theta_2 H_{\mathbb C^n}^2$, and hence
$$
\Delta \Theta_0 H_{\mathbb C^n}^2=\Delta \Theta_0 H_{\mathbb C^n}^2
\bigcap \Theta_2 H_{\mathbb C^n}^2 \subseteq \Theta_0 H_{\mathbb C^n}^2
\bigcap \Theta_2 H_{\mathbb C^n}^2=\Theta_0 \Theta_2 H_{\mathbb C^n}^2,
$$
which implies $\Delta H_{\mathbb C^n}^2\subseteq \Theta_2 H_{\mathbb
C^n}^2$. \ This proves (\ref{3.17}). \ Therefore we must have that
$\Delta^{\prime}$ is a unitary constant, so that $A$ and $\Theta_2$
are (right) coprime. \ This completes the proof.
\end{proof}

\bigskip

%%%%%%%%%%%%%%%%%%%%%%%%%%%%%%%%%%%%%%%%%%%%%%%%%%%%%%%%%%%%%%%%%%%%%%%%%%%%%%%%
%
%
%                Section 2
%
%
%%%%%%%%%%%%%%%%%%%%%%%%%%%%%%%%%%%%%%%%%%%%%%%%%%%%%%%%%%%%%%%%%%%%%%%%%%%%%%%

\section{The proof of Theorem \ref{thm2002}}

\medskip

In this section we give a proof of Theorem \ref{thm2002}. \ In the
sequel, and for notational convenience we will denote the model
space generated by inner function $\Theta \in H^{\infty}_{M_n}$ by
$\mathcal H(\Theta)$, i.e.,
$$
\mathcal H(\Theta):=H^2_{\mathbb C^n}\ominus \Theta H^2_{\mathbb
C^n}.
$$
It is known that (cf. \cite[Corollary A.15]{CHL5})
$$
f\in \mathcal H(\Theta) \Longleftrightarrow f\in H^2_{\mathbb C^n} \
\hbox{and} \ \Theta^* f \in (H^2_{\mathbb C^n})^{\perp}.
$$

 To prove Theorem \ref{thm2002}, we need three
auxiliary lemmas.

\begin{lem}\label{lem3.1} Let $\Phi \in L^{\infty}_{M_n}$ be such that $T_{\Phi}$
is a hyponormal operator and let $K\in \mathcal E(\Phi)$. \ Then
$$
[T_{\Phi}^*,
T_{\Phi}]=H_{\Phi^*}^*(I-T_{\widetilde{K}}T_{\widetilde{K}^*})H_{\Phi^*},
$$
where $\widetilde{K}(z)=K^*(\overline{z})$.
\end{lem}
\begin{proof} This readily follows from (\ref{basic}).
\end{proof}

\begin{lem}\label{lem3.2} Let $\Phi$ be a normal matrix-valued
function in $L^{\infty}_{M_n}$. \ If there is an inner function
$\Theta$ in $\mathcal E(\Phi)$, then the closure of $\hbox{Ran}\,
[T_{\Phi}^*, T_{\Phi}]$ equals the closure of
$T_{\Phi\Theta^*}\mathcal H(\Theta)$.
\end{lem}
\begin{proof} Let $\Theta$ be an inner matrix function in $\mathcal
E(\Phi)$. \ Put $G:=\Phi-\Theta \Phi^*$. \ Then $G \in H^{\infty}_{M_n}$
and we may write
$$
\Phi=\Theta \Phi^*+G.
$$
We will show that $\hbox{Ran}\,[T_{\Phi}^*, T_{\Phi}]$ is dense in
$T_{\Phi\Theta^*}\mathcal H(\Theta)$. \ First, we will prove that
\begin{equation}\label{10101}
\hbox{Ran}\,[T_{\Phi}^*, T_{\Phi}]
=H_{\Phi^*}^*(I-T_{\widetilde{\Theta}}T_{\widetilde{\Theta}}^*)
H_{\Phi^*}H^2_{\mathbb
C^n}
\subseteq T_{\Phi\Theta^*}\mathcal H(\Theta).
\end{equation}

To do so, we prove the following two claims.

\medskip

\noindent{\bf Claim 1.}
$
H_{\Phi^*}^*(I-T_{\widetilde{\Theta}}T_{\widetilde{\Theta}}^*)H^2_{\mathbb
C^n}=H_{\Phi^*}^*\mathcal H(\widetilde{\Theta}).
$

\medskip
\noindent For a proof of Claim 1, observe that
$$
I-T_{\widetilde{\Theta}}T_{\widetilde{\Theta}}^*=H_{\Theta^*}H_{{\Theta}^*}^*,
$$
and hence
$$
\ker(I-T_{\widetilde{\Theta}}T_{\widetilde{\Theta}}^*)=\ker
H_{\Theta^*}^*=\ker
H_{\widetilde{\Theta}^*}=\widetilde{\Theta}H^2_{\mathbb C^n}.
$$
Thus
$$
(I-T_{\widetilde{\Theta}}T_{\widetilde{\Theta}}^*)H^2_{\mathbb
C^n}=\mathcal H(\widetilde{\Theta}),
$$
and therefore
$$
H_{\Phi^*}^*(I-T_{\widetilde{\Theta}}T_{\widetilde{\Theta}}^*)H^2_{\mathbb
C^n}=H_{\Phi^*}^* \mathcal H(\widetilde{\Theta}).
$$

\medskip

\noindent{\bf Claim 2.}
$
H_{\Phi^*}^*\mathcal H(\widetilde{\Theta})=T_{\Phi \Theta^*}
\mathcal H(\Theta).
$

\medskip
\noindent For a proof of Claim 2, let $f\in H_{\Phi^*}^* \mathcal
H(\widetilde{\Theta})$, and hence $f=H_{\Phi^*}^* g$ for some $g\in
\mathcal H(\widetilde{\Theta})$. \ We observe that
$$
f=H_{\Phi^*}^*g=H_{\breve{\Phi}} g=J_n P_n^{\perp}
M_{\breve{\Phi}}g=P_n M_{\Phi}J_n g.
$$
Now take $g^{\prime}:=\Theta(J_n g)$. \ Since $g\in \mathcal
H(\widetilde{\Theta})$, we have $\breve{\Theta}g \in (H^2_{\mathbb
C^n})^{\perp}$. \ Thus
$$
g^{\prime}=\Theta(J_n g)=J_n(\breve{\Theta}g)\in H^2_{M_n} \quad
\hbox{and} \quad \Theta^* g^{\prime}=J_n g\in (H^2_{\mathbb
C^n})^{\perp},
$$
which implies that $g^{\prime} \in \mathcal H(\Theta)$. \ Observe that
$$
T_{\Phi \Theta^*}g^{\prime}=P_n(\Phi\Theta^* \Theta(J_n
g))=P_n(\Phi(J_ng))=P_nM_{\Phi}J_n g=f.
$$
Thus we have
$$
H_{\Phi^*}^* \mathcal H(\widetilde{\Theta}) \subseteq
T_{\Phi\Theta^*} \mathcal H(\Theta).
$$
Conversely, let $f\in T_{\Phi \Theta^*} \mathcal H(\Theta)$, and
hence $f=T_{\Phi \Theta^*}g$ for some $g\in \mathcal H(\Theta)$. \ Let
$g^{\prime}:=J_n(\Theta^* g)$. \ Then $g^{\prime} \in H^2_{\mathbb
C^n}$ and $\widetilde{\Theta}^*g^{\prime}=\breve{\Theta}J_n
(\Theta^*g)=J_n g \in (H^2_{\mathbb C^n})^{\perp}$. \ Thus $g^{\prime}
\in \mathcal H(\widetilde{\Theta})$ and
$$
\aligned
H_{\Phi^*}^* g^{\prime}&=H_{\Phi^*}^*J_n(\Theta^* g)=J_n
P_n^{\perp}M_{\breve{\Phi}}J_n(\Theta^* g)\\
&=P_n M_{\Phi}J_nJ_n(\Theta^* g)=P_n(\Phi\Theta^*
g)\\
&=T_{\Phi \Theta^*}g=f.
\endaligned
$$
Therefore,
$
T_{\Phi \Theta^*} \mathcal H(\Theta) \subseteq H_{\Phi^*}^* \mathcal
H(\widetilde{\Theta}),
$
which proves Claim 2.

\medskip

Now by combining Claim 1 and Claim 2 , we have
$$
\aligned \hbox{Ran}\,[T_{\Phi}^*,
T_{\Phi}]&=H_{\Phi^*}^*(I-T_{\widetilde{\Theta}}T_{\widetilde{\Theta}}^*)H_{\Phi^*}H^2_{\mathbb
C^n}\\
&\subseteq
H_{\Phi^*}^*(I-T_{\widetilde{\Theta}}T_{\widetilde{\Theta}}^*)H^2_{\mathbb
C^n}\\
&=T_{\Phi\Theta^*}\mathcal H(\Theta),
\endaligned
$$
which proves (\ref{10101}).

\medskip

We next prove that $\hbox{Ran}\,[T_{\Phi}^*, T_{\Phi}]$ is dense
in $T_{\Phi \Theta^*} \mathcal H(\Theta)$.
To do so , we first prove

\medskip

\noindent{\bf Claim 3.}
$
\bigl((I-T_{\widetilde{\Theta}}T_{\widetilde{\Theta}}^*)
H_{\Phi^*}\bigr)^*(I-T_{\widetilde{\Theta}}T_{\widetilde{\Theta}}^*)
H_{\Phi^*}=H_{\Phi^*}^*(I-T_{\widetilde{\Theta}}T_{\widetilde{\Theta}}^*)
H_{\Phi^*}.
$

\medskip
\noindent For a proof of Claim 3, observe that
$$
\aligned
(I-T_{\widetilde{\Theta}}T_{\widetilde{\Theta}}^*)(I-T_{\widetilde{\Theta}}T_{\widetilde{\Theta}}^*)&=I-T_{\widetilde{\Theta}}T_{\widetilde{\Theta}}^*
-T_{\widetilde{\Theta}}T_{\widetilde{\Theta}}^*+T_{\widetilde{\Theta}}T_{\widetilde{\Theta}}^*T_{\widetilde{\Theta}}T_{\widetilde{\Theta}}^*\\
&=I-2T_{\widetilde{\Theta}}T_{\widetilde{\Theta}}^*+T_{\widetilde{\Theta}}T_{\widetilde{\Theta}}^*\\
&=I-T_{\widetilde{\Theta}}T_{\widetilde{\Theta}}^*.
\endaligned
$$
Hence
$$
\bigl((I-T_{\widetilde{\Theta}}T_{\widetilde{\Theta}}^*)
H_{\Phi^*}\bigr)^*(I-T_{\widetilde{\Theta}}T_{\widetilde{\Theta}}^*)
H_{\Phi^*}=H_{\Phi^*}^*(I-T_{\widetilde{\Theta}}T_{\widetilde{\Theta}}^*)
H_{\Phi^*},
$$
which proves Claim 3.

Thus it follows from Claim 3 that
\begin{equation}\label{10102}
\ker [T_{\Phi}^*, T_{\Phi}]=\ker
H_{\Phi^*}^*(I-T_{\widetilde{\Theta}}T_{\widetilde{\Theta}}^*)H_{\Phi^*}=\ker
(I-T_{\widetilde{\Theta}}T_{\widetilde{\Theta}}^*)H_{\Phi^*}
\end{equation}

\medskip

We next prove

\medskip

\noindent{\bf Claim 4.}
$
H_{\Phi^*}^*(I-T_{\widetilde{\Theta}}T_{\widetilde{\Theta}}^*)H^2_{\mathbb
C^n} \subseteq (\ker[T_{\Phi}^*, T_{\Phi}])^{\perp} = \hbox{cl
Ran}\, [T_{\Phi}^*, T_{\Phi}].
$

\medskip
\noindent To establish Claim 4, let $f \in
H_{\Phi^*}^*(I-T_{\widetilde{\Theta}}T_{\widetilde{\Theta}}^*)H^2_{\mathbb
C^n} $ and $g \in \ker[T_{\Phi}^*, T_{\Phi}]$. \ Thus,
$f=H_{\Phi^*}^*(I-T_{\widetilde{\Theta}}T_{\widetilde{\Theta}}^*)h$ for some
$h\in H^2_{\mathbb C^n}$. \ Observe
$$
\aligned \langle f, g\rangle&=\langle
H_{\Phi^*}^*(I-T_{\widetilde{\Theta}}T_{\widetilde{\Theta}}^*)h, g
\rangle\\
&=\langle h,
(I-T_{\widetilde{\Theta}}T_{\widetilde{\Theta}}^*)H_{\Phi^*} g \rangle\\
&=0,
\endaligned
$$
because, by (\ref{10102}), $\ker [T_{\Phi}^*, T_{\Phi}] =\ker
(I-T_{\widetilde{\Theta}}T_{\widetilde{\Theta}}^*)H_{\Phi^*}$. \ Thus
$f \in (\ker [T_{\Phi}^*, T_{\Phi}])^{\perp}$. \ Now, recall that for
any bounded self-adjoint operator $A$ on a Hilbert space $\mathcal
H$, we have $\ker A=(\hbox{Ran}\, A)^{\perp}$ and $(\ker
A)^{\perp}=\hbox{cl Ran}\, A$. \ Therefore,
$$
H_{\Phi^*}^*(I-T_{\widetilde{\Theta}}T_{\widetilde{\Theta}}^*)H^2_{\mathbb
C^n} \subseteq (\ker [T_{\Phi}^*, T_{\Phi}])^{\perp} =\hbox{cl
Ran}\, [T_{\Phi}^*, T_{\Phi}],
$$
which proves Claim 4.

\medskip

By Claim 1 and Claim 2, we have
$H_{\Phi^*}^*(I-T_{\widetilde{\Theta}}T_{\widetilde{\Theta}}^*)H^2_{\mathbb
C^n}=T_{\Phi \Theta^*} \mathcal H(\Theta)$. \ Thus it follows from
(\ref{10101}) and Claim 4 that
$$
\hbox{Ran}\, [T_{\Phi}^*, T_{\Phi}]\subseteq T_{\Phi \Theta^*}
\mathcal
H(\Theta)=H_{\Phi^*}^*(I-T_{\widetilde{\Theta}}T_{\widetilde{\Theta}}^*)H^2_{\mathbb
C^n}\subseteq \hbox{cl Ran}\, [T_{\Phi}^*, T_{\Phi}].
$$
Therefore, the closure of $\hbox{Ran}\, [T_{\Phi}^*, T_{\Phi}]$ is
equal to the closure of $T_{\Phi \Theta^*} \mathcal H(\Theta)$. \ This
completes the proof.
\end{proof}

\bigskip

\begin{lem}\label{thm3.3} If $T_{\Phi}$ is hyponormal and there is a finite
Blaschke-Potapov product $Q$ in $\mathcal E(\Phi)$, then
$[T_{\Phi}^*, T_{\Phi}]$ is a finite rank operator.
\end{lem}

\begin{proof} Observe that
$$
[T_{\Phi}^*,
T_{\Phi}]=H_{\Phi^*}^*(I-T_{\widetilde{Q}}T_{\widetilde{Q}}^*)H_{\Phi^*}.
$$
We thus have
$$
\hbox{rank}\,[T_{\Phi}^*, T_{\Phi}]\leq
\hbox{rank}\,(I-T_{\widetilde{Q}}T_{\widetilde{Q}}^*) =
\hbox{rank}\, H_{Q^*}=\dim \mathcal H(Q)<\infty.
$$
\end{proof}

We are ready to prove Theorem \ref{thm2002}.

\medskip

\begin{proof}[Proof of Theorem \ref{thm2002}]\
Suppose $\Phi=Q\Phi^*$, where $Q$ is a finite Blaschke-Potapov
product. \ Then by Lemma \ref{thm3.3}, $\hbox{Ran}\,[T_{\Phi}^*,
T_{\Phi}]$ is finite dimensional. \ By Lemma \ref{lem3.2}, we
have
\begin{equation}\label{3.1}
\hbox{Ran}\,[T_{\Phi}^*, T_{\Phi}]=T_{\Phi Q^*} \mathcal
H(Q)=T_{\Phi^*} \mathcal H(Q).
\end{equation}
If $Q$ is a unitary matrix then $\mathcal H(Q)=\{0\}$, and hence
$\hbox{Ran}\,[T_{\Phi}^*, T_{\Phi}]=\{0\}$, which says that
$T_{\Phi}$ is normal. \ Suppose instead $Q$ is not a unitary matrix,
so that $\dim \mathcal H(Q)\geq 1$. \ Since $[T_{\Phi}^*, T_{\Phi}]$
is a self-adjoint finite rank operator, we have $\hbox{Ran}\,[T_{\Phi}^*,
T_{\Phi}]=(\ker [T_{\Phi}^*, T_{\Phi}])^{\perp}$. \ Recall that if
$T_{\Phi}$ is subnormal then $\hbox{Ran}\,[T_{\Phi}^*, T_{\Phi}]$ is
an invariant subspace for $T_{\Phi^*}$. \ Thus it follows from
(\ref{3.1}) that
\begin{equation}\label{3.2}
T_{\Phi^*}(\hbox{Ran}\,[T_{\Phi}^*, T_{\Phi}]+\mathcal
H(Q))=\hbox{Ran}\,[T_{\Phi}^*, T_{\Phi}].
\end{equation}
Suppose $\Phi^*$ is not of bounded type. \ Then $\overline{\Phi}$ is
not of bounded type. \ We now claim that
\begin{equation}\label{3.3}
\hbox{Ran}\,[T_{\Phi}^*, T_{\Phi}] \bigcap \mathcal H(Q)=\{0\}.
\end{equation}
To see this, let $f\in \hbox{Ran}\,[T_{\Phi}^*, T_{\Phi}] \cap
\mathcal H(Q)$. \ By (\ref{3.1}), $f=T_{\Phi Q^*} g$ for some $g
\in \mathcal H(Q)$. \ Since $f \perp Q H^2_{\mathbb C^n}$, it follows
$Q^* f \perp H^2_{\mathbb C^n}$, and hence $T_{Q^*} f=0$. \ But since
$\Phi Q^*=Q^* \Phi$, we have that
$$
0=T_{Q^*} f=T_{Q^*}T_{\Phi Q^*} g=T_{Q^{*2}\Phi}g,
$$
which implies $Q^{*2}\Phi g=\Phi Q^{*2}g \perp H^2_{\mathbb C^n}$.
Therefore, $(g^*Q^2 \Phi^*)^t=\overline{\Phi}{Q^{2}}^t \overline{g}
\in H^2_{\mathbb C^n}$, so that ${Q^2}^t \overline{g}\in \ker
H_{\overline{\Phi}}$. \ Since by assumption, $H_{\overline{\Phi}}$ is
injective it follows ${Q^2}^t \overline{g}=0$. \ Since $Q Q^*=I$, it
follows $\overline{g}=0$, and hence $g=0$. \ Hence $f=T_{\Phi Q^*}
g=0$. \ This proves (\ref{3.3}). \ Using (\ref{3.3}), we have
$$
\dim\bigl(\hbox{Ran}\, [T_{\Phi}^*, T_{\Phi}]+\mathcal
H(Q)\bigr)=\dim \hbox{Ran}\, [T_{\Phi}^*, T_{\Phi}]+\dim \mathcal
H(Q).
$$
Then, by the Rank Theorem,
$$
\aligned \dim \ &\hbox{Ran}\, [T_{\Phi}^*, T_{\Phi}]+\dim\mathcal
H(Q)\\
&=\dim \ker T_{\Phi^*}|_{(\hbox{Ran}\, [T_{\Phi^*},
T_{\Phi}]+\mathcal H(Q))}+\dim \hbox{Ran}\,
T_{\Phi^*}|_{(\hbox{Ran}\, [T_{\Phi^*}, T_{\Phi}]+\mathcal H(Q))}\\
&=\dim \ker T_{\Phi^*}|_{(\hbox{Ran}\, [T_{\Phi^*},
T_{\Phi}]+\mathcal H(Q))}+\dim \hbox{Ran}\, [T_{\Phi^*}, T_{\Phi}]
\quad(\hbox{by} \ (\ref{3.2})),
\endaligned
$$
which gives
$$
\dim \mathcal H(Q)=\dim \ker T_{\Phi^*}|_{(\hbox{Ran}\, [T_{\Phi}^*,
T_{\Phi}]+\mathcal H(Q))}
$$
Thus we have
$$
\dim \ker T_{\Phi^*} \geq \dim \mathcal H(Q) \geq 1,
$$
which leads to a contradiction because $T_{\Phi^*}$ is injective by
assumption. \ This completes the proof.
\end{proof}

\bigskip

\begin{rem}\label{rem3.4} In general, if $H_{\overline{\Phi}}$ is not
injective then we don't guarantee that \linebreak $\hbox{Ran}\,[T_{\Phi}^*,
T_{\Phi}] \bigcap \mathcal H(Q)=\{0\}$. \ To see this, take a function $f
\in H^{\infty}$ such that $\overline{f}$ is not of bounded type and
put $\phi=f +\overline{f}$. \ Take
$$
\Phi=\begin{bmatrix}\phi&0\\0&z\end{bmatrix}.
$$
Then $\Phi^*$ is not
of bounded type and $\Phi=Q \Phi^*$ with
$Q=\begin{bmatrix}1&0\\0&z^2\end{bmatrix}$. \ Note that $Q$ is a
finite Blaschke-Potapov product. \ Then
$$
H_{\overline{\Phi}}\begin{bmatrix}0\\z
\end{bmatrix}=\begin{bmatrix}
H_{\phi}&0\\0&H_{\overline{z}}\end{bmatrix}\begin{bmatrix}0\\z\end{bmatrix}=\begin{bmatrix}0\\0\end{bmatrix},
$$
which says that $H_{\overline{\Phi}}$ is not injective. \ But since
$[T_{\Phi}^*,
T_{\Phi}]\begin{bmatrix}0\\1\end{bmatrix}=\begin{bmatrix}0\\1\end{bmatrix}$
and
$$
\left \langle \begin{bmatrix}0\\1\end{bmatrix}, \
Q\begin{bmatrix}f\\g\end{bmatrix} \right \rangle_{H^2_{\mathbb
C^2}}=\left \langle \begin{bmatrix}0\\1\end{bmatrix}, \
\begin{bmatrix}f\\z^2g\end{bmatrix} \right \rangle_{H^2_{\mathbb
C^2}}=0 \quad \hbox{for all} \ f, g \in H^2,
$$
it follows that $\begin{bmatrix}0\\1\end{bmatrix} \in \hbox{Ran}\,
[T_{\Phi}^*, T_{\Phi}]\cap \mathcal H(Q)$.
\end{rem}

\begin{rem}\label{rem3.5} In Theorem \ref{thm2002}, the injectivity of
$T_{\Phi^*}$ and $H_{\overline{\Phi}}$ are not necessary even though
$T_{\Phi}$ is self-adjoint. \ To see this, take a function $\phi$ as
in Remark \ref{rem3.4}. \ Since $\phi$ is real-valued, it follows
$T_{\phi}$ is self-adjoint. \ Take
$$
\Phi=\begin{bmatrix}\phi&0\\0&0\end{bmatrix}.
$$
Then $\Phi^*$ is not
of bounded type and $T_{\Phi}$ is self-adjoint. \ Observe $\Phi=Q
\Phi^*$ with $Q=\begin{bmatrix}1&0\\0&z\end{bmatrix}$. \ Note that $Q$
is a finite Blaschke-Potapov product. \ Observe that
$$
T_{\Phi^*}\begin{bmatrix}0\\1\end{bmatrix}=\begin{bmatrix}T_{\overline{\phi}}&0\\0&0\end{bmatrix}\begin{bmatrix}0\\1\end{bmatrix}=\begin{bmatrix}0\\0\end{bmatrix}
$$
and
$$
H_{\overline{\Phi}}\begin{bmatrix}0\\1\end{bmatrix}=\begin{bmatrix}H_{\overline{\phi}}&0\\0&0\end{bmatrix}\begin{bmatrix}0\\1\end{bmatrix}=\begin{bmatrix}0\\0\end{bmatrix},
$$
which show that both $T_{\Phi^*}$ and $H_{\overline{\Phi}}$ are not
injective.
\end{rem}

\vskip 1cm

\noindent \textit{Acknowledgments}. \ The authors are grateful to the Referee for a detailed reading of the paper and for helpful suggestions. \ The first-named author
(Abhinand M.) was supported by the Junior Research Fellowship,
University Grant commission, Government of India. \ The third-named
author (I.S. Hwang) was supported by NRF(Korea) grant No.
2022R1A2C1010830. \ The fourth-named author (W.Y. Lee) was supported
by NRF(Korea) grant No. 2021R1A2C1005428. \ The fifth-named author
(Prasad T.) was supported in part by the Mathematical Research
Impact Centric Support, MATRICS (MTR/2021/000373) by SERB,
Department of Science and Technology (DST), Government of India.

\bigskip

%%%%%%%%%%%%%%%%%%%%%%%%%%%%%%%%%%%%%%%%%%%%%%%%%%%%%%%%%%%%%%%%%%%%%%%%%%
%
%
%            References
%
%
%%%%%%%%%%%%%%%%%%%%%%%%%%%%%%%%%%%%%%%%%%%%%%%%%%%%%%%%%%%%%%%%%%%%%%%%%%

%%%%%%%%%%%%%%%%%%%%%%%%%%%%%%%%%%%%%%%%%%%%%%%%%%%%%%%%%%%%%%%%%%%%%%%%%%%%%%

\vskip 1cm

Mankunikuzhiyil Abhinand

Department of Mathematics, University of Calicut, Kerala-673635,
India

E-mail: abhinandmkrishnan@gmail.com

\bigskip

Ra{\'u}l\ E.\ Curto

Department of Mathematics, University of Iowa, Iowa City, IA 52242,
U.S.A.

E-mail: raul-curto@uiowa.edu

\bigskip

In Sung Hwang

Department of Mathematics, Sungkyunkwan University, Suwon 16419,
Korea

E-mail: ihwang@skku.edu

\bigskip

Woo Young Lee

School of Mathematics, Korea Institute for Advanced
Study (KIAS), Seoul 02455, Korea

E-mail: wylee@kias.re.kr

\bigskip

Thankarajan Prasad

Department of Mathematics, University of Calicut, Kerala-673635,
India

E-mail: prasadvalapil@gmail.com

\end{document}